\documentclass[12pt]{amsart}

\usepackage{amssymb}
\usepackage{bm}
\usepackage{graphicx}
\usepackage{psfrag}
\usepackage{color}
\usepackage{hyperref}
\hypersetup{colorlinks=true, linkcolor=blue, citecolor=red, urlcolor=wine}
\usepackage{url}
\usepackage{algpseudocode}
\usepackage{fancyhdr}
\usepackage{xy}
\input xy
\xyoption{all}

\usepackage{soul} 

\newtheorem*{maintheorem*}{Main Theorem}
\newtheorem{theorem}{Theorem}[section]
\newtheorem{proposition}[theorem]{Proposition}
\newtheorem{lemma}[theorem]{Lemma}
\newtheorem{cor}[theorem]{Corollary}

\theoremstyle{definition}
\newtheorem{definition}[theorem]{Definition}
\newtheorem{example}[theorem]{Example}
\newtheorem{remark}[theorem]{Remark}
\numberwithin{equation}{section}

\newcommand\nn{\mathbb{N}}
\newcommand\qq{\mathbb{Q}}
\newcommand\rr{\mathbb{R}}
\newcommand\zz{\mathbb{Z}}
\newcommand\pval{\mathsf{v}_p}

\keywords{Archimedean field, atomicity, BF-monoid, FF-monoid, ordered field, positive monoid, Puiseux monoid}


\begin{document}

	\mbox{}
	\title{Increasing Positive Monoids of \\ Ordered Fields Are FF-monoids}
	\author{Felix Gotti}
	\address{Mathematics Department\\UC Berkeley\\Berkeley, CA 94720}
	\email{felixgotti@berkeley.edu}
	\date{\today}
	
	\begin{abstract}
		Given an ambient ordered field $K$, a positive monoid is a countably generated additive submonoid of the nonnegative cone of $K$. In this paper, we first generalize several atomic features exhibited by Puiseux monoids of the field of rational numbers to the more general setting of positive monoids of Archimedean fields, accordingly arguing that such generalizations might fail if the ambient field is not Archimedean. In particular, we show that a positive monoid $P$ of an Archimedean field is a BF-monoid provided that $P \! \setminus \! \{0\}$ does not have $0$ as a limit point. Then, we prove our main result: every increasing positive monoid of an ordered field is an FF-monoid. Finally, we deduce that every increasing positive monoid is hereditarily atomic.
	\end{abstract}
	
	\maketitle
	
	\section{Introduction} \label{sec:intro}
	
	The family of Puiseux monoids was introduced in \cite{fG16}, where the atomic structure of its members was studied. \emph{Puiseux monoids} are additive submonoids of $\qq_{\ge 0}$. They exhibit a very complex atomic structure. Indeed, there are nontrivial Puiseux monoids having no irreducible elements at all (i.e., being \emph{antimatter}), while others, failing to be atomic, contain infinitely many irreducible elements.
	
	In this paper, we generalize the notion of Puiseux monoid of $\qq$ by considering certain additive submonoids of the nonnegative rational cone of an arbitrary ordered field. In fact, we will study the atomic structure of an even more general family of commutative monoids (see Section~\ref{sec:ordered fields} for the definitions related to ordered fields).
	
	\begin{definition}
		Let $K$ be an ordered field. A \emph{positive monoid} $P$ of $K$ is a countably generated additive submonoid of the nonnegative cone of $K$. In this case, we call $K$ an \emph{ambient field} for $P$.
	\end{definition}
	
	 Every Puiseux monoid is, therefore, a positive monoid of the ambient field $\qq$. In \cite{fG16} and \cite{GG16}, many techniques were introduced to understand the atomic structure of Puiseux monoids. Here we will modify various of these results, providing the appropriate conditions for them to hold in the more general context of positive monoids of an arbitrary ordered field. Furthermore, we study the family of positive monoids that can be generated by increasing sequences, the following being our main theorem.
	 
	 \begin{maintheorem*}
	 	Every increasing positive monoid of an ordered field is an \emph{FF}-monoid.
	 \end{maintheorem*}
	 
	 After verifying that every submonoid of a positive BF-monoid is atomic, we obtain, as a consequence of our main result, that if a positive monoid is increasing, then all its submonoids are atomic.
	
	 This paper is organized in the following way. In Section~\ref{sec:commutative monoids}, we establish the nomenclature of commutative monoids as well as the terminology for the basic notions related to their atomic structure and factorization theory. Then, in Section~\ref{sec:ordered fields}, we go over the fundamental concepts of ordered fields we will be using throughout this paper. In Section~\ref{sec:from Puiseux monoids to positive monoids}, among other minor results, we describe those positive monoids that are isomorphic to Puiseux monoids. In the same section, we also present Proposition~\ref{prop: positive moniods in Archimedean fields are atomic if they don't have zero as a limit point} and Proposition~\ref{prop:strongly increasing iff hereditarily increasing}, two results on Puiseux monoids that naturally generalize to positive monoids of Archimedean ordered fields, but that fail when the Archimedean condition is dropped. In Section~\ref{sec:increasing positive monoids}, we study the atomic structure of increasing positive monoids. We will argue that a positive monoid $P$ is a BF-monoid provided that $P^\bullet$ does not have~$0$ as a limit point. In addition, we deduce from our main theorem that every increasing positive monoid is hereditarily atomic. The last section is dedicated to prove our main result.
	
	\section{Atomicity and Factorization on Commutative Monoids} \label{sec:commutative monoids}
	
	This section contains basic terminology concerning the atomic structure and factorization theory of commutative monoids. Here we also introduce notation for two special families of commutative monoids that will appear systematically in this sequel: numerical semigroups and Puiseux monoids. Reference material on commutative semigroups can be found in \cite{pG01} of Grillet. On the other hand, the monograph \cite{GH06} of Geroldinger and Halter-Koch offers extensive background information on atomic monoids and their non-unique factorization theory.
	
	We use the double-struck symbol $\mathbb{N}$ to denote the set of positive integers, and we set $\nn_0 = \mathbb{N} \cup \{0\}$. If $R \subseteq \rr$ and $r \in R$, then we let $R_{\ge r}$ denote the set $\{x \in R \mid x \ge r\}$. With a similar intension we use $R_{\le r}$, $R_{> r}$, and $R_{< r}$. If $q \in \qq_{> 0}$, then the unique $a,b \in \nn$ such that $q = a/b$ and $\gcd(a,b)=1$ are denoted by $\mathsf{n}(q)$ and $\mathsf{d}(q)$, respectively. For $Q \subseteq \qq_{>0}$, the sets
	\[
		\mathsf{n}(Q) = \{\mathsf{n}(q) \mid q \in Q\} \quad \text{and} \quad \mathsf{d}(Q) = \{\mathsf{d}(q) \mid q \in Q\}
	\] are called \emph{numerator set} and \emph{denominator set} of $Q$, respectively. Finally, for a set $S$ we will write sometimes $\{s_n\} \in S^\infty$ when $\{s_n\}$ is a sequence whose terms are in $S$.
	
	Recall that a \emph{semigroup} is just a map $* \colon S \times S \to S$, which is usually denoted by $(S,*)$, or simply by $S$ provided that $*$ is clear from the context. A semigroup with identity is customary called a \emph{monoid}. However, to avoid wordiness, every \emph{monoid} in this paper is always presumed commutative and cancellative. In addition, every monoid homomorphism here is assumed to send the identity to the identity. For the rest of this section, let $M$ be a monoid.
	
	Because every monoid here is assumed to be commutative, we will use additive notation. The set $M \! \setminus \! \{0\}$ is denoted by $M^\bullet$\!, while the set of units of $M$ is denoted by $M^\times$\!. The monoid $M$ is said to be \emph{reduced} if $M^\times$ contains only the identity element. All monoids we will be dealing with are reduced. For $a,c \in M$, we say that $a$ \emph{divides}~$c$ \emph{in} $M$ and write $a \mid_M c$ if $c = a + b$ for some $b \in M$. A submonoid $N$ of $M$ is said to be \emph{divisor-closed} if for every $a \in N$ and $d \in M$ the fact that $d \mid_M a$ implies that $d \in N$. We write $M = \langle S \rangle$ when $M$ is generated by a set $S$. The monoid $M$ is \emph{finitely generated} if it can be generated by a finite set. A succinct exposition of finitely generated commutative monoids can be found in \cite{GR99} by Garc\'ia-S\'anchez and Rosales.
	
	An element $a \in M \! \setminus \! M^\times$ is \emph{irreducible} or an \emph{atom} if $a = x + y$ for $x,y \in M$ implies that either $x \in M^\times$ or $y \in M^\times$. The set of atoms of $M$ is denoted by $\mathcal{A}(M)$, and $M$ is called \emph{atomic} if $M = \langle \mathcal{A}(M) \rangle$. By contrast, $M$ is said to be \emph{antimatter} if $\mathcal{A}(M)$ is empty. Antimatter domains and monoids were first defined in \cite{CDM99} and \cite{fG16}, respectively.
	
	 Assume that $M$ is reduced. The free abelian monoid on $\mathcal{A}(M)$ is denoted by $\mathsf{Z}(M)$ and called \emph{factorization monoid} of $M$; the elements of $\mathsf{Z}(M)$ are called \emph{factorizations}. If $z = a_1 \dots a_n \in \mathsf{Z}(M)$ for some $n \in \nn_0$ and $a_1, \dots, a_n \in \mathcal{A}(M)$, then $n$ is the \emph{length} of the factorization $z$, commonly denoted by $|z|$; we say that an atom $a$ \emph{shows in} $z$ if $a \in \{a_1, \dots, a_n\}$. The unique homomorphism
	 \[
		 \phi \colon \mathsf{Z}(M) \to M \ \ \text{satisfying} \ \ \phi(a) = a \ \ \text{for all} \ \ a \in \mathcal{A}(M)
	 \]
	 is called the \emph{factorization homomorphism} of $M$. For $x \in M^\bullet$,
	 \[
		 \mathsf{Z}(x) = \phi^{-1}(x) \subseteq \mathsf{Z}(M)
	\]
	is the \emph{set of factorizations} of $x$. By definition, we set $\mathsf{Z}(0) = \{0\}$. If $x \in M$ satisfies $|\mathsf{Z}(x)| = 1$, then we say that $x$ has \emph{unique factorization}. Note that the monoid $M$ is atomic if and only if $\mathsf{Z}(x)$ is not empty for all $x \in M$.
	
	\begin{definition}
		A monoid $M$ is called an \emph{FF-monoid} if it is atomic and for all $x \in M$ the set $\mathsf{Z}(x)$ is finite.
	\end{definition}
	
	The next proposition, which we will refer recurrently here, follows from \cite[Proposition~2.7.8(4)]{GH06a}.
	
	\begin{proposition} \label{prop:f.g. atomic monoids are FF-monoids}
		Every finitely generated reduced monoid is an \emph{FF}-monoid.
	\end{proposition}
	
	\noindent For each $x \in M$, the \emph{set of lengths} of $x$ is defined by
	\[
		\mathsf{L}(x) = \{|z| : z \in \mathsf{Z}(x)\}.
	\]
	The set of lengths and further arithmetic invariants of many families of atomic monoids has been significantly studied during the last decades (see \cite{ACHP07,CCMMP14,CGP14,GS16} and references therein). If $\mathsf{L}(x)$ is a finite set for all $x \in M$, we say that $M$ satisfies the \emph{bounded factorization property}, in which case, we call $M$ a BF-\emph{monoid}. Proposition~\ref{prop:f.g. atomic monoids are FF-monoids} implies that every finitely generated atomic monoid is a BF-monoid.
	
	A \emph{numerical semigroup} is a cofinite submonoid of the additive monoid $\nn_0$. Every numerical semigroup has a unique minimal set of generators, which is finite. For a numerical semigroup $N$ minimally generated by the positive integers $a_1 , \dots, a_n$, we have that $\gcd(a_1, \dots, a_n) = 1$ and $\mathcal{A}(N) = \{a_1, \dots, a_n\}$. Thus, every numerical semigroup is an atomic monoid containing finitely many atoms. A great first approach to the realm of numerical semigroups can be found in \cite{GR09}.
	
	\vspace{3pt}
	A \emph{Puiseux monoid} is an additive submonoid of $\qq_{\ge 0}$. Albeit a natural generalization of numerical semigroups, Puiseux monoids are not always atomic. Consider, for instance, $\langle 1/2^n \mid n \in \nn \rangle$. However, a Puiseux monoid $M$ is atomic provided $M^\bullet$ does not have $0$ as a limit point (\cite[Theorem~3.10]{fG16}). If an atomic Puiseux monoid is not isomorphic to a numerical semigroup, then it contains infinitely many atoms. For example, it is not hard to verify that if $M = \langle 1/p \mid p \text{ is prime} \rangle$, then $|\mathcal{A}(M)| = \infty$. The atomicity of Puiseux monoids was the center of attention in \cite{fG16}.
	
	\begin{definition}
		A Puiseux monoid is \emph{increasing} (resp., \emph{decreasing}) if it can be generated by an increasing (resp., decreasing) sequence of rationals.
	\end{definition}

	Let $M$ be a Puiseux monoid. We say that $M$ is \emph{monotone} if it is either increasing or decreasing. On the other hand, we say that $M$ is \emph{strongly increasing} if there exists an increasing generating sequence $\{s_n\}$ for $M$ such that $\lim s_n = \infty$. The atomic structure of monotone Puiseux monoids was studied in \cite{GG16}.
	
	\section{Ordered Fields} \label{sec:ordered fields}
	
	In this section, we briefly recall some concepts related to ordered fields as a way to establish the notation we will be using later. For ordered fields we mostly follow the notation in \cite{DW96}. In addition, in \cite[Chapters 11 and 12]{sL02}, readers can find the rudiments on ordered fields we will assume in this sequel.
	
	\begin{definition}
		An \emph{ordered field} is a pair $(K,K^+)$, where $K$ is a field and $K^+ \subset K$, satisfying the following conditions.
		\begin{itemize}
			\item For all $x,y \in K^+$,  $x+y\in K^+$ and $xy \in K^+$.
			\item For all $x \in K \! \setminus \! \{0\}$, exactly one of the statements $x=0$, $x \in K^+$ and $-x \in K^+$ holds.
		\end{itemize}
		In this case, we also say that $K$ is an ordered field with \emph{positive cone} $K^+$ (or \emph{nonnegative cone} $K^+ \cup \{0\}$).
	\end{definition}
	
	\begin{remark}
		Notice that if $K$ is an ordered field with positive cone $K^+$, then $K$ is also an ordered set, where $x \le y$ in $K$ if and only if $y-x \in K^+$. We will always consider $K^+$ as an ordered set with the order $\le$ inherited from $K$.
	\end{remark}
	
	Let $K$ be an ordered field with positive cone $K^+$. Clearly, every ordered field has characteristic zero. Therefore the prime subfield of $K$ is isomorphic to $\qq$. Indeed, the prime subfield of every ordered field in this paper will always be identified with $\qq$. For each $x \in K$ set $|x| = x$ if $x \in K^+ \cup \{0\}$ and $|x| = -x$ otherwise. In addition, for $y \in K$, we write $x = {\bf O}(y)$ if $|x| \le n|y|$ for some $n \in \nn$, and $x \sim y$ if both $x = {\bf O}(y)$ and $y = {\bf O}(x)$ hold. Clearly, $\sim$ defines an equivalence relation on $K^\times$. Let
	\[
		 \beta \colon K^\times \to \Gamma_K := K^\times/\!\sim
	\]
	be the quotient map. Setting $\beta(x) \preceq \beta(y)$ when $y = {\bf O}(x)$, one finds that $(\Gamma_K, \preceq)$ is a well-defined totally ordered set. Moreover, the multiplication of $K$ induces a group structure on $\Gamma_K$ under which $\Gamma_K$ is a totally ordered group (see \cite{DW96}). The group $\Gamma_K$ is the \emph{value group} of $K$. The elements of $\Gamma_K$ are called \emph{Archimedean classes}, and the quotient map $\beta \colon K^\times \to \Gamma_K$ is called \emph{Archimedean valuation}.
	  
	 An element $x \in K$ is \emph{finite} if $x = {\bf O}(1)$, while $x$ is called \emph{infinitesimal} (resp., \emph{infinitely large}) if $|x| < 1/n$ (resp., $|x| > n$) for every natural number $n$. Obviously, $K$ contains nonzero infinitesimals if and only if it contains infinitely large elements. The set of infinitesimals of $K$ is denoted by $K_0$, while the set of finite elements is denoted by $K_\#$; they are both additive subgroups of $K$. The field $K$ is said to be \emph{Archimedean} if $K_0 = \{0\}$. Note that $K$ is Archimedean if and only if its value group $\Gamma_K$ is trivial; readers can find $42$ equivalent definitions of Archimedean ordered field in \cite[Section~4]{DT14}.
	 
	 The order topology on $K$ has a basis consisting of all the intervals of the forms $(x,y)$, $(-\infty, y)$, and $(x, \infty)$ for $x,y \in K$. It is well-known that $K$ is Archimedean if and only if its prime subfield is order-theoretically dense in $K$ and, in such a case, $K$ is isomorphic as an ordered field to a subfield of $\rr$. Moreover, $K$ is a Hausdorff topological group (under addition) and a completely regular space (see \cite[Lemma~2.1]{dD00}). The field $K$ is said to be \emph{complete} if every Cauchy sequence converges. Every ordered field can be densely order-embedded in a complete ordered field. There are many equivalent definitions of completeness; $72$ of them are given in \cite[Section 3]{DT14}.
	 
	 We conclude this section presenting an example of non-Archimedean ordered field that we will be using later.
	 
	 \begin{example} \label{ex:non-Archimedean field of fractions}
	 	Let $K$ be an ordered field, and let $K(X)$ be the field of rational functions over $K$. If $p(X) \in K[X]$ is a nonzero polynomial, then let $\ell(p(X))$ denote its leading coefficient. Now set
	 	\begin{align} \label{eq:positive cone of the field of fractions}
		 	K(X)^+ = \bigg\{\frac{p(X)}{q(X)} \ \bigg{|} \ p(X), q(X) \in K[X] \! \setminus \! \{0\} \ \text{ and } \ \frac{\ell(p(X))}{\ell(q(X))} > 0 \bigg\}
	 	\end{align}
	 	and check that $K(X)^+$ is indeed a positive cone making $K(X)$ an ordered field. The ordered field $K(X)$ is not Archimedean as $1/X$ (resp., $X$) is infinitesimal (resp., infinitely large). For another non-Archimedean ordering on $K(X)$, see \cite[Example~2.5]{HT15}. In this paper, we always consider $K(X)$ as an ordered field with the positive cone given in \eqref{eq:positive cone of the field of fractions}.
	 \end{example}

	\section{From Puiseux Monoids to Positive Monoids} \label{sec:from Puiseux monoids to positive monoids}
	
	Recall that a positive monoid $P$ of a given ambient ordered field $K$ is a countably generated additive submonoid of the nonnegative cone of $K$. We begin this section describing the positive monoids that are isomorphic to Puiseux monoids. Then we restate two properties of Puiseux monoids, \cite[Theorem~3.10]{fG16} and \cite[Theorem~3.9]{GG16}, but in the more general context of positive monoids of an Archimedean ordered field, and we verify that these results do not hold when the ambient field fails to be Archimedean. First, let us generalize the concept of Puiseux monoid.
	
	\begin{definition}
		Let $K$ be an ordered field. A \emph{Puiseux monoid} of $K$ is a positive monoid that is contained in the prime subfield of $K$.
	\end{definition}
	
	Since we can always identify the prime subfield of an ordered field with the field of rational numbers, our definition of Puiseux monoid is consistent with that one given in Section \ref{sec:commutative monoids}. It follows immediately that a Puiseux monoid is isomorphic to a numerical semigroup if and only if it is finitely generated. It is natural to wonder under which circumstances a positive monoid is isomorphic to a Puiseux monoid. Notice that if~$P$ is a positive monoid of an ordered field $K$, then so is $aP$ for all $a \in K^+$. The next proposition classifies the positive monoids that are isomorphic to either Puiseux monoids or numerical semigroups.
	
	\begin{proposition}
		Let $K$ be an ordered field, and let $P$ be a positive monoid of $K$. Then the following statements hold.
		\begin{enumerate}
			\item $P$ is isomorphic to a Puiseux monoid of $K$ if and only if there exists $a \in K_{> 0}$ such that $aP$ is a Puiseux monoid.
			\item $P$ is isomorphic to a numerical semigroup if and only if $P$ is finitely generated and $aP$ is a Puiseux monoid for some $a \in K_{> 0}$\!.
		\end{enumerate}
	\end{proposition}
	
	\begin{proof}
		To show (1), suppose that $P$ is isomorphic to a Puiseux monoid $Q = \langle q_n \mid n \in \nn \rangle$ via the isomorphism $\varphi \colon Q \to P$, where $\{q_n\}$ is a sequence of positive rationals. The submonoid $N = \nn_0 \cap Q$ is finitely generated, say $N = \langle n_1, \dots, n_k \rangle$ for some $k \in \nn_0$ and $n_1, \dots, n_k \in \nn$. For every $i \in \{1,\dots,k\}$ we have
		\[
			\varphi(n_i) = \frac{1}{n_1} \varphi(n_1n_i) = \frac{n_i}{n_1} \varphi(n_1).
		\]
		Since $\varphi$ is injective, $\varphi(n_1) \neq 0$. Set $a = n_1/\varphi(n_1)$. If $q \in Q^\bullet$, then $\mathsf{n}(q) \in N$. Therefore there exist coefficients $c_1, \dots, c_k \in \nn_0$ such that $\mathsf{n}(q) = c_1n_1 + \dots + c_kn_k$. As a result, one obtains that
		\[
			\mathsf{d}(q) \varphi(q) = \varphi(\mathsf{n}(q)) = \varphi\bigg(\sum_{i=1}^k c_i n_i\bigg) = \sum_{i=1}^k c_i \varphi(n_i) = \sum_{i=1}^k c_i n_ia^{-1} = a^{-1} \mathsf{n}(q).
		\]
		Thus, $\varphi(q) = a^{-1}q$ for every $q \in Q$. Since $\varphi$ is surjective $a^{-1}Q = P$, which means that $aP$ is the Puiseux monoid $Q$.
		
		Conversely, suppose that $aP$ is a Puiseux monoid for some $a \in K^+$. Since multiplication by $a$ defines an isomorphism from $P$ to $aP$, it follows immediately that $P$ is isomorphic to a Puiseux monoid.
			
		Now we verify (2). If $P$ is isomorphic to a numerical semigroup, then it is finitely generated. Since every numerical semigroup is, in particular, a Puiseux monoid, $P$ is isomorphic to a Puiseux monoid. By part (1), there exists $a \in K^+$ such that $aP$ is a Puiseux monoid. Finally, let us check the reverse implication of (2). Since $aP$ is a Puiseux monoid for some $a \in K^+$, part (1) ensures that $P$ is isomorphic to a Puiseux monoid. Because finitely generated Puiseux monoids are isomorphic to numerical semigroups, the proof is complete.
	\end{proof}
	
	The next three propositions establish sufficient conditions for positive and Puiseux monoids to be atomic.
	
	\begin{proposition} \label{prop:a PM is atomic iff it has a minimal generating set}
		Let $P$ be a positive monoid of an ordered field. Then $P$ contains a minimal generating set $A$ if and only if $P$ is atomic with $A = \mathcal{A}(P)$; in such a case, $A$ is the unique minimal generating set of $P$.
	\end{proposition}
	
	The above proposition follows from the fact that every positive monoid of an ordered field is reduced (see \cite[Proposition~1.1.7]{GH06a}).
	
	Let $K$ be an ordered field. As we have identified the prime subfield of $K$ with~$\qq$, it makes sense to talk about prime, natural, and integer numbers in $K$. For a prime $p$, recall that the $p$-\emph{adic valuation} on $\qq$ is the map defined by $\pval(0) = \infty$ and $\pval(a/b) = \pval(a) - \pval(b)$ for all nonzero integers $a$ and $b$, where $\pval(z)$ is the exponent of the maximal power of $p$ dividing the integer number $z$. We say that a Puiseux monoid~$P$ of $K$ is \emph{finite} if there is a finite subset $S$ of $\qq_{>0}$ consisting of prime numbers such that $\pval(x) \ge 0$ for every $x \in P^\bullet$ and $p \notin S$. It is not hard to argue the following proposition.

	\begin{proposition}
		Let $P$ be a Puiseux monoid of an ordered field. Then $P$ is finite and $\{\pval(P)\}$ is bounded from below for every prime $p$ if and only if $\mathsf{d}(P^\bullet)$ is bounded. Moreover, if one of these conditions holds, then $P$ is atomic.
	\end{proposition}
	
	Let $K$ be an Archimedean ordered field. For every $x \in K$ there exists a unique integer~$n_x$ such that $n_x \le x < n_x + 1$. Hence the floor and ceiling functions make sense in the context of Archimedean fields. The next proposition generalizes \cite[Theorem~3.10]{fG16}, which says that if $0$ is not a limit point of a Puiseux monoid $P$ of $\rr$, then $P$ is atomic.
	
	\begin{proposition} \label{prop: positive moniods in Archimedean fields are atomic if they don't have zero as a limit point}
		Let $P$ be a positive monoid of an Archimedean ordered field. If $0$ is not a limit point of $P^\bullet$, then $P$ is a \emph{BF}-monoid.
	\end{proposition}
	
	\begin{proof}
		Let $K$ be an ambient ordered field of $P$. Clearly, the set $\mathcal{A}(P)$ consists of those elements of $P^\bullet$ that cannot be written as the sum of two positive elements of $P$. Since~$0$ is not a limit point of $P^\bullet$ there exists $\epsilon \in K_{> 0}$ such that $\epsilon < x$ for all $x \in P^\bullet$. To show that $P$ is atomic (i.e., $P = \langle \mathcal{A}(P) \rangle$), take $x \in P^\bullet$. Since $\epsilon$ is a lower bound for $P^\bullet$ and $\lfloor x/\epsilon \rfloor + 1 > x/\epsilon$, the element $x$ can be written as the sum of at most $\lfloor x/\epsilon \rfloor$ elements of $P^\bullet$. Let $m$ be the maximum natural number such that $x = a_1 + \dots + a_m$ for some $a_1, \dots, a_m \in P^\bullet$. By the maximality of $m$, it follows that $a_i \in \mathcal{A}(P)$ for each $i = 1,\dots,m$, which means that $x \in \langle \mathcal{A}(P) \rangle$. Hence $P$ is atomic. We have already noticed that every element $x$ in $P^\bullet$ can be written as the sum of at most $\lfloor x/\epsilon \rfloor$ positive elements, i.e., $|\mathsf{L}(x)| \le \lfloor x/\epsilon \rfloor$ for all $x \in P$. Thus, $P$ is a BF-monoid.
	\end{proof}
	
	As a special version of our main theorem, we have that every increasing positive monoid of an Archimedean field is an FF-monoid. However, under the hypothesis of Proposition~\ref{prop: positive moniods in Archimedean fields are atomic if they don't have zero as a limit point} we cannot always guarantee that $P$ is an FF-monoid. The next example illustrates this observation.
	
	\begin{example}
		Let $\{p_n\}$ be an enumeration of the prime numbers. Consider the Puiseux monoid $P$ of the Archimedean ambient field $\rr$ generated by the set
		\[
			A = \bigg \{ \frac{p_n + \lfloor p_n/2 \rfloor}{p_n}, \frac{2p_n - \lfloor p_n/2 \rfloor}{p_n}  \ \bigg{|} \ n \in \nn \bigg \}.
		\]
		Since $1 < a < 2$ for every $a \in A$, it follows that $\mathcal{A}(P) = A$ and, therefore, $P$ is atomic. As $a > 1$ for each $a \in A$, we see that $0$ is not a limit point of $P^\bullet$. However, for every $n \in \nn$, one finds that
		\[
			3 = \frac{p_n + \lfloor p_n/2 \rfloor}{p_n} + \frac{2p_n - \lfloor p_n/2 \rfloor}{p_n},
		\]
		which implies that $\mathsf{Z}(3)$ contains infinitely many factorizations. Hence $P$ fails to be an FF-monoid.
	\end{example}
	
	As we work on the more general setting of positive monoids of an arbitrary ordered field, the potential inclusion of infinitesimals might cause the failure of some properties showing in the more particular scenario of Puiseux monoids of Archimedean fields. For instance, let us see that the Archimedean condition in Proposition~\ref{prop: positive moniods in Archimedean fields are atomic if they don't have zero as a limit point} is required. Let~$K$ be a non-Archimedean ordered field, and let $\epsilon$ be a positive infinitesimal of $K$. So $\epsilon \le r$ for all $r$ in the positive cone $\qq_{>0}$ of the prime subfield. Since $\qq_{>0} \cap (-\epsilon, \epsilon)$ is empty, $0$ is not a limit point of the positive monoid $\qq_{>0}$. However, notice that $\qq_{>0}$ is not atomic; indeed, $\qq_{> 0}$ is antimatter because every $q \in \qq_{>0}$ is divisible by $q/2$.
	
	We say that a positive monoid $P$ of an ordered field $K$ is \emph{strongly increasing} provided that there exists an increasing generating sequence $\{s_n\}$ for $P$ such that for each $x \in K^+$ one can find $n \in \nn$ with $s_n > x$ (cf. definition of strongly increasing Puiseux monoid at the end of Section~\ref{sec:commutative monoids}). Note that this definition depends on the ambient field we choose to embed the positive monoid. It was proved in \cite[Section 3]{GG16} that a Puiseux monoid $P$ of $\qq$ is strongly increasing if and only if every submonoid of $P$ is increasing; the proof given there can be mimicked to establish the proposition below.
	
	\begin{proposition} \label{prop:strongly increasing iff hereditarily increasing}
		Let $P$ be a positive monoid of an Archimedean ordered field. Then $P$ is strongly increasing if and only if every submonoid of $P$ is increasing.
	\end{proposition}
	
	The above proposition yields another property holding for positive monoids of an Archimedean ordered field that will no longer be true if we drop the Archimedean condition. In fact, both implications might fail if the ambient ordered field is not Archimedean. The next two examples shed light on this observation.
	
	\begin{example}
		Let $\rr(X)$ be the field of rational functions regarded as an ordered field with the positive cone given in Example~\ref{ex:non-Archimedean field of fractions}. Consider its positive monoids
		\[
			P = \big\langle X^n \mid n \in \nn \big\rangle \ \text{ and } \ P' = \big \langle X^3, X + nX^2 \mid n \in \nn \big \rangle.
		\]
		Note that $P$ is strongly increasing in $\rr(X)$ and $P'$ is a submonoid of $P$. We verify now that $P'$ is not increasing. Since $X^3 \notin \langle X + nX^2 \mid n \in \nn \rangle$, it is an atom of $P'$. Take $n \in \nn$ and $\alpha, \alpha_1, \dots, \alpha_n \in \nn_0$ such that
		\[
			X + nX^2 = \alpha X^3 + \sum_{k=1}^n \alpha_kX + \alpha_k k X^2
		\]
		(note that, in the above sum, it is enough to add only up to $n$). The displayed polynomial equality forces $\alpha = 0$ and $\alpha_1 + \dots + \alpha_n = 1$. Thus, there exists $j \in \{1, \dots, n\}$ such that $\alpha_j = 1$ and $\alpha_i = 0$ for $i \neq j$. This implies that $X + nX^2 \in \mathcal{A}(P')$. Therefore $\mathcal{A}(P') = \{X^3\} \cup \{X + nX^2 \mid n \in \nn\}$. Since there are infinitely many elements in $\mathcal{A}(P')$ that are less than $X^3$, the set of atoms of $P'$ is not the underlying set of any increasing sequence. Hence $P'$ cannot be generated by any increasing sequence. As a consequence, $P'$ fails to be an increasing positive monoid.
	\end{example}
	
	\begin{example}
		Consider again the field of rational functions $\rr(X)$ with the ordering given in Example~\ref{ex:non-Archimedean field of fractions}, and let $\{p_n\}$ be an increasing enumeration of the prime numbers considered as elements of the prime subfield of $\rr(X)$. In addition, let
		\begin{equation} \label{eq:classic strongly increasing PM}
			P = \bigg \langle \frac{p_n^2 + 1}{p_n} \ \bigg{|} \ n \in \nn \bigg\rangle.
		\end{equation}
		It is clear that $P$ is a strongly increasing Puiseux monoid of the ambient field $\qq$. So it follows by \cite[Theorem~3.9]{GG16} that every submonoid of $P$ is increasing in $\qq$ and, therefore, in $\rr(X)$. However, $P$ is not strongly increasing as a positive monoid of $\rr(X)$; this is because $X$ is an upper bound for $P$. Hence a positive monoid might fail to be strongly increasing even when all its submonoids are increasing.
	\end{example}

	\subsection*{Appendix}\footnote{The appendix subsection is not part of the original paper, as published in Journal of Algebra. However, note that the appendix has been added in a way that it does not affect the original labeling of definitions, examples, theorem, etc.}
	
	The converse of Proposition~\ref{prop: positive moniods in Archimedean fields are atomic if they don't have zero as a limit point} does not hold even in the context of Puiseux monoids. The next example confirms this fact.
	
	\begin{example}
		Let $\{p_n\}$ and $\{q_n\}$ be two strictly increasing sequence of prime numbers such that $q_n > p_n^2$. Now consider the Puiseux monoid
		\[
			P = \bigg \langle \frac{p_n}{q_n} \ \bigg{|} \ n \in \nn \bigg \rangle.
		\]
		It follows from~\cite[Corollary~5.6]{GG16} that $P$ is an atomic monoid, and it can be readily verified that $\mathcal{A}(P) = \{p_n/q_n \mid n \in \nn\}$. As $q_n > p_n^2$ for every $n \in \nn$, we have that~$0$ is a limit point of $P^\bullet$. Let us proceed to argue that $P$ is indeed an FF-monoid. Take $x \in P^\bullet$. Since both sequences $\{p_n\}$ and $\{q_n\}$ are strictly increasing, there exists $N \in \nn$ such that $q_n \nmid \mathsf{d}(x)$ and $p_n > x$ for every $n \ge N$. Now if $z \in \mathsf{Z}(x)$, then only the atoms $p_1/q_1, \dots, p_N/q_N$ can appear in $z$. Moreover, for each $i = 1, \dots, N$, the factorization~$z$ can contain at most $\big\lfloor \frac{q_i}{p_i}x \big\rfloor$ copies of the atom $p_i/q_i$. As a result, $\mathsf{Z}(x)$ is a finite set. Hence $P$ is an FF-monoid and, in particular, a BF-monoid. Note that the fact that $P$ is an FF-monoid cannot be deduced via the main result of this paper because $P$ is not an increasing Puiseux monoid.
	\end{example}
	
	\section{Increasing Positive Monoids} \label{sec:increasing positive monoids}
	
	In this section we extend several results achieved in \cite{GG16} to positive monoids of more general ordered fields. First, let us extend the concept of \emph{monotonicity} from Puiseux monoids to positive monoids. 
	
	\begin{definition}
		A positive monoid of an ordered field is \emph{increasing} (resp., \emph{decreasing}) if it can be generated by an increasing (resp., decreasing) sequence. A positive monoid of an ordered field is called \emph{monotone} if it is either increasing or decreasing.
	\end{definition}
	
	 If $P$ is an increasing positive monoid of an ordered field $K$, then $P$ is atomic and if $\{a_n\}$ is an increasing sequence generating $P$, then $\mathcal{A}(P) = \{a_n \mid a_n \notin \langle a_1, \dots, a_{n-1} \rangle \}$. This was proved in \cite[Proposition 3.2]{GG16} for $K = \qq$; the proof given there applies, \emph{mutatis mutandis}, when $K$ is an arbitrary ordered field. Let us record this observation to facilitate further references.
	
	\begin{proposition} \label{prop:increasing are atomic}
		Let $K$ be an ordered field, and let $P$ be the positive monoid of $K$ generated by the increasing sequence $\{a_n\}$. Then $P$ is atomic and
		\[
			\mathcal{A}(P) = \big\{a_n \mid a_n \notin \langle a_1, \dots, a_{n-1} \rangle \big\}.
		\]
	\end{proposition}
	
	We say that a subset $S$ of an ordered field is \emph{increasing} (resp., \emph{decreasing}) if it is the underlying set of an increasing (resp., decreasing) sequence. If $S$ is either increasing or decreasing, then we say that it is \emph{monotone}. Clearly, any monotone subset of an ordered field must be (at most) countable.
	
	\begin{lemma} \label{lem:properties of monotone sets}
		A subset of an ordered field is both increasing and decreasing if and only if it is finite.
	\end{lemma}
	
	\begin{proof}
		Let $K$ be an ordered field, and let $S$ be a subset of $K$. Suppose first that~$S$ is increasing and decreasing. The fact that $S$ is decreasing implies that $S$ has a maximum element, namely, the first element of any decreasing sequence of $K$ with underlying set~$S$. Notice now that every increasing sequence with underlying set $S$ must stabilize at $\max S$. Hence $S$ is finite. On the other hand, if $S$ is finite, then it is increasing and decreasing; this is because we can increasingly (resp., decreasingly) enumerate the elements of $S$ as the first $|S|$ elements of a sequence and then complete the rest of the sequence taking copies of $\max S$ (resp., $\min S$).
	\end{proof}
	
	 Let $P$ be a positive monoid of some ambient ordered field. By Lemma~\ref{lem:properties of monotone sets}, if $P$ is finitely generated, then it is increasing and decreasing. On the other hand, suppose that $P$ is both increasing and decreasing. By Proposition~\ref{prop:increasing are atomic}, one finds that $P$ is atomic. Since $\mathcal{A}(P)$ is contained in every generating set, it is increasing and decreasing. Lemma~\ref{lem:properties of monotone sets} now implies that $\mathcal{A}(P)$ is finite and, therefore, $P$ is finitely generated. Hence the next result holds.
	 
	 \begin{proposition} \label{prop:f.g. iff increasing and decreasing}
	 	A positive monoid of an ordered field is finitely generated if and only if it is increasing and decreasing.
	 \end{proposition}

	A positive monoid of an ordered field $K$ is \emph{weakly increasing} if it can be generated by a bounded increasing sequence of $K$. A weakly increasing positive monoid is, in particular, increasing. A Puiseux monoid of $\qq$ is both strongly and weakly increasing if and only if it is isomorphic to a numerical semigroup (see \cite[Proposition~3.7]{GG16}). This fact does not extend to positive monoids of an arbitrary ordered field, as the next proposition indicates.
	
	\begin{proposition} \label{prop: f.g. implies strongly increasing iff Archimedean}
		Let $K$ be an ordered field, and let $P$ be a finitely generated positive monoid of $K$. Then $P$ is strongly increasing if and only if $K$ is Archimedean.
	\end{proposition}
	
	\begin{proof}
		For the forward implication suppose, by way of contradiction, that $K$ is not Archimedean. Take $k \in \nn$ and $a_1, \dots, a_k \in K^+$ such that $P = \langle a_1, \dots, a_k \rangle$ and $0 < a_1 < \dots < a_k$. If $a_k$ were finite, then $P \subseteq K_\#$ and, therefore, any infinitely large element in $K^+$ would be an upper bound for $P$, a contradiction. Thus, assume that $a_k$ is infinitely large. In this case, for all coefficients $n_1, \dots, n_k \in \nn_0$,
		\[
			\sum_{i=1}^k n_i a_i \le kNa_k < a_k^2,
		\]
		where $N = \max\{n_1, \dots, n_k\}$. Since $a_k^2$ is an upper bound for $P$, it follows that $P$ is not strongly increasing, which is a contradiction. Hence $K$ must be Archimedean.
		
		For the reverse implication assume that $K$ is Archimedean. Once again, take $k \in \nn$ and $a_1, \dots, a_k \in K^+$ such that $0 < a_1 < \dots < a_k$ and $P = \langle a_1, \dots, a_k \rangle$. Define the sequence $\{u_n\}$ of $K$ by setting $u_i = a_i$ for $i \in \{1, \dots, k\}$ and $u_n = n a_k$ for $n > k$. The sequence $\{u_n\}$ is increasing and generates $P$. Since $K$ is Archimedean, for every $x \in K^+$ there exists $n \in \nn$ such that $u_n = n a_k > x$. As $P$ is generated by the unbounded increasing sequence $\{u_n\}$, it is a strongly increasing positive monoid, which completes the proof.
	\end{proof}

	We have seen in Proposition~\ref{prop: positive moniods in Archimedean fields are atomic if they don't have zero as a limit point} that a positive monoid $P$ of an Archimedean ordered field is a BF-monoid provided that $P^\bullet$ does not have $0$ as a limit point. Strengthening the hypothesis of this result, we can guarantee that $P$ is, in fact, an FF-monoid. This is the main result of this paper.
	
	\begin{theorem} \label{thm:increasing PM of Archimedean fields are FF}
		Every increasing positive monoid of an ordered field is an \emph{FF}-monoid.
	\end{theorem}
	
	The increasing condition in Theorem~\ref{thm:increasing PM of Archimedean fields are FF} is not superfluous, as we now illustrate by providing an example of an atomic Puiseux monoid that is not even a BF-monoid.
	
	\begin{example} \label{ex:decreasing PM that is not BF}
		Let $\{p_n\}$ be an increasing enumeration of the prime numbers. In the ambient field $\qq$, consider the Puiseux monoid
		\[
			P = \langle A \rangle, \ \text{ where } \ A = \bigg \{ \frac 1{p_n} \ \bigg{|} \ n \in \nn \bigg \}.
		\]
		It is easy to check that $P$ is an atomic monoid with $\mathcal{A}(P) = A$. Since $A$ is infinite and~$P$ is decreasing, Proposition~\ref{prop:f.g. iff increasing and decreasing} implies that $P$ is not increasing. In addition, $P$ is not a BF-monoid as $1$ is the sum of $p_n$ copies of the atom $1/p_n$ for every natural number $n$. In particular, $P$ is not an FF-monoid.
	\end{example}
	
	On the other hand, the converse of Theorem~\ref{thm:increasing PM of Archimedean fields are FF} is not true; the following example sheds light upon this observation.
	
	\begin{example} \label{ex:FF PM that is not increasing}
		Let $\{p_n\}$ be a strictly increasing sequence of primes, and consider the Puiseux monoid of $\qq$ defined as follows:
		\begin{equation} \label{eq:FF does not implies increasing}
			P = \langle A \rangle, \ \text{where} \ A = \bigg\{\frac{p_{2n}^2 + 1}{p_{2n}}, \frac{p_{2n+1} + 1}{p_{2n+1}} \ \bigg{|} \ n \in \nn \bigg\}.
		\end{equation}
		Since $A$ is an unbounded subset of $\rr$ having $1$ as a limit point, it cannot be increasing. In addition, as $\mathsf{d}(a) \neq \mathsf{d}(a')$ for all $a,a' \in A$ such that $a \neq a'$, every element of $A$ is an atom of $P$. Thus, every generating set of $P$ must contain $A$. Now the fact that $A$ is not increasing implies that $P$ is not an increasing positive monoid.
		
		We verify now that $P$ is an FF-monoid. Fix $x \in P$ and then take $D_x$ to be the set of prime numbers dividing $\mathsf{d}(x)$. Now choose a natural number $N$ large enough such that $N > \max \{x, \mathsf{d}(x)\}$. For each $a \in A$ such that $\mathsf{d}(a) > N$, the number of copies $\alpha$ of the atom $a$ showing in any $z \in \mathsf{Z}(x)$ must be a multiple of $\mathsf{d}(a)$ because $\mathsf{d}(a) \notin D_x$. Therefore $\alpha = 0$; otherwise, $x \ge \alpha a \ge \mathsf{d}(a)a > \mathsf{d}(a) > x$. Thus, if an atom $a$ divides $x$ in $P$, then $\mathsf{d}(a) \le N$. As a result, only finitely many elements of $\mathcal{A}(P)$ divide $x$ in $P$. This implies that $\mathsf{Z}(x)$ is finite. Hence $P$ is an FF-monoid that fails to be increasing.
	\end{example}
	\medskip
	 
	 \section{Proof of the Main Theorem} \label{sec:proof of main theorem}
	 
	 In this section, we prove our main result, Theorem~\ref{thm:increasing PM of Archimedean fields are FF}. First, let us verify two technical results that are crucial for its proof.
	 
	 \begin{lemma} \label{lem:f.g. ordered monoids do not contain strictly decreasing sequences}
	 	A finitely generated positive monoid does not contain a strictly decreasing sequence.
	 \end{lemma}
	 
	 \begin{proof}
	 	 Assume, by way of contradiction, that there exists a nonempty family $\mathcal{F}$ of finitely generated positive monoids containing strictly decreasing sequences. Among the members of $\mathcal{F}$ take a positive monoid $P$ such that $|\mathcal{A}(P)| = \min \{|\mathcal{A}(F)| : F \in \mathcal{F}\}$. Let $K$ be an ambient ordered field for $P$. By Proposition~\ref{prop:a PM is atomic iff it has a minimal generating set}, one has that $P$ is atomic. Let $\mathcal{A}(P) = \{a_1, \dots, a_m\}$, where $m \in \nn$ and $a_1 < \dots < a_m$. Also, take $\{s_n\}$ to be a strictly decreasing sequence of $K^+$ contained in $P$. For every $n \in \nn$ write
	 	 \[
		 	 s_n = \alpha_{n1}a_1 + \dots + \alpha_{nm}a_m
		 \]
		 for some $\alpha_{ij} \in \nn_0$. If for $\alpha \in \nn$, there is a strictly increasing sequence of natural numbers $\{k_n\}$ such that $\alpha_{k_n1} = \alpha$, then taking $s'_n = s_{k_n}$ we would find that $\{s'_n - \alpha a_1\}$ is a strictly decreasing sequence contained in the finitely generated positive monoid $\langle a_2, \dots, a_n \rangle$, contradicting the minimality of $|\mathcal{A}(P)|$. As a result, $\lim_{n \to \infty} \{\alpha_{n1}\} = \infty$. Similarly, we can argue that $\lim_{n \to \infty} \alpha_{nj} = \infty$ for each $j=2,\dots,m$. This implies that there exists a natural number $N>1$ such that $\alpha_{Nj} > \alpha_{1j}$ for each $j=2,\dots,m$. As a result, $s_N > s_1$, which contradicts the fact that $\{s_n\}$ is decreasing.
	 \end{proof}
	 
	  If $M$ is an atomic monoid and $N$ is an atomic submonoid of $M$, then for $x \in N$ the set $\mathsf{Z}(x)$ depends on whether we consider $x$ as an element in $M$ or $N$. The same is true for the set $\mathsf{L}(x)$. When there is some risk of confusion, we write $\mathsf{Z}_M(x)$ (resp., $\mathsf{Z}_N(x)$) to refer the factorization set of $x$ in $M$ (resp., $N$). We use the notations $\mathsf{L}_M(x)$ and $\mathsf{L}_N(x)$ with the same intension. Recall that the submonoid $N$ of $M$ is called divisor-closed if for every $a \in N$ and $d \in M$ the fact that $d \mid_M a$ implies that $d \in N$. 
	  
	  \begin{lemma} \label{lem:union of nested divisor-closed FF is FF}
	  	Let $M$ be an atomic reduced monoid, and let $\{M_n\}$ be a sequence of divisor-closed submonoids of $M$ such that
	  	\[
		  	M = \bigcup_{n \in \nn} M_n.
	  	\]
	  	If every $M_n$ is an \emph{FF}-monoid, then $M$ is also an \emph{FF}-monoid.
	  \end{lemma}
	  
	  \begin{proof}
		  Since $M$ is reduced, $M_n$ is reduced for each $n \in \nn$. Let $x$ be an element of~$M$. Since $M$ is the union of the $M_n$'s, we have that $x \in M_n$ for some $n \in \nn$. Take $z \in \mathsf{Z}_M(x)$. Since $M_n$ is divisor-closed, every atom of $M$ showing in $z$ belongs to $M_n$. The fact that $\mathcal{A}(M) \cap M_n \subseteq \mathcal{A}(M_n)$ now implies that $z \in \mathsf{Z}_{M_n}(x)$. Consequently, $\mathsf{Z}_M(x) \subseteq \mathsf{Z}_{M_n}(x)$. Since $M_n$ is an FF-monoid, the set of factorizations $\mathsf{Z}_{M_n}(x)$ is finite, which implies that $\mathsf{Z}_M(x)$ is also finite. Since $x$ was arbitrarily taken, it follows that~$M$ is an FF-monoid.
	 \end{proof}

	Let $M$ be a reduced atomic monoid, and let $A$ be a subset of $\mathcal{A}(M)$. For $z \in \mathsf{Z}(M)$, we let $|z|_A$ denote the number of atoms in $A$ showing in $z$ (counting repetition). Note that $|\cdot|$ and $|\cdot|_A$ are the same if and only if $A = \mathcal{A}(M)$. Since $\mathsf{Z}(M)$ is free on $\mathcal{A}(M)$, there exists a unique monoid homomorphism
	\[
		\phi_A \colon \mathsf{Z}(M) \to M
	\]
	such that $\phi_A(a) = a$ if $a \in A$ and $\phi_A(a) = 0$ if $a \in \mathcal{A}(M) \setminus A$; we call $\phi_A$ the \emph{factorization homomorphism restricted to} $A$.
	
	 We are now in a position to prove our main result. \\
	
	\noindent {\it Proof of Theorem~\ref{thm:increasing PM of Archimedean fields are FF}:}
		Let $K$ be an ordered field, and let $P$ be an increasing positive monoid of $K$. Since $P$ is increasing, it must be atomic by Proposition~\ref{prop:increasing are atomic}. If $P$ is finitely generated, then it follows by Proposition~\ref{prop:f.g. atomic monoids are FF-monoids} that $P$ is an FF-monoid. Therefore let us assume that $P$ is not finitely generated, that is, $|\mathcal{A}(P)| = \infty$. Take $\{a_n\}$ to be a strictly increasing sequence of $K$ with underlying set $\mathcal{A}(P)$. Let $\beta \colon K^\times \to \Gamma_K$ be the Archimedean valuation of $K$ (see Section~\ref{sec:ordered fields}). Because $\{a_n\}$ increases, it follows that $\beta(a_{n+1}) \preceq \beta(a_n)$ for every $n \in \nn$.
		
		We show first that $P$ is an FF-monoid when the set $\{\beta(a_n) \mid n \in \nn\}$ of Archimedean classes is finite. Let us assume, by way of contradiction, that $P$ is not an FF-monoid. Choose $x \in P^\bullet$ such that $\mathsf{Z}(x)$ contains infinitely many factorizations. Take the minimum $N \in \nn$ such that $\beta(a_n) = \beta(a_m)$ for all $n,m \ge N$. Now set
		\[
			A = \{a_j \mid j \ge N\}.
		\]
		Since $a + a' = {\bf O}(\max\{a,a'\})$ for all $a,a' \in K^+$ and each element $y \in P^\bullet$ can be written as the sum of atoms, it follows that $\beta(a_N) \preceq \beta(y)$ for all $y \in P$. As a result, there exists a smallest positive integer $N'$ such that $N' a_N \ge x$. If for some $j \ge N$ the atom~$a_j$ shows in infinitely many factorizations of $x$, we can replace $x$ by $x - a_j$ and still preserve the fact that $|\mathsf{Z}(x)| = \infty$. Such a replacement can happen at most $N'$ times; otherwise, there would exist a sequence $\{c_n\} \in \nn_0^\infty$ with $\sum_{j \ge N} c_j \ge N'$ such that
		\[
			x \ge \sum_{j \ge N} c_j a_j \ge \sum_{j \ge N} c_j a_N > N' a_N \ge x.
		\]
		Therefore we can assume that for every $j \ge N$ the atom $a_j$ shows in only finitely many factorizations in $\mathsf{Z}(x)$. Since $|\mathsf{Z}(x)| = \infty$ and every factorization in $\mathsf{Z}(x)$ contains at most $N'$ copies of atoms in $A$, there exists $n_0 \le N'$ such that the set
		\[
			Z = \{z \in \mathsf{Z}(x) : |z|_A = n_0\}
		\]
		is infinite. As for every $j \ge N$ the atom $a_j$ shows in only finitely many factorizations of $x$, we can construct a sequence of factorizations $\{z_n\}$ in $Z$ such that $\{\phi_A(z_n)\}$ is a strictly increasing sequence of $P$: take $z_1 \in Z$ arbitrarily and, once we have constructed $\{z_1, \dots, z_{n-1}\}$ such that $\phi_A(z_1) < \dots < \phi_A(z_{n-1})$, take $z_n \in Z$ such that every atom of $A$ showing in $z_n$ is strictly greater than the maximum atom showing in $z_{n-1}$. As any two factorizations in $Z$ have the same number of atoms contained in $A$, one sees that $\phi_A(z_{n-1}) < \phi_A(z_n)$. Therefore $\{x - \phi_A(z_n)\}$ is a strictly decreasing sequence in $\langle a_1, \dots, a_{N-1} \rangle$, which contradicts Lemma~\ref{lem:f.g. ordered monoids do not contain strictly decreasing sequences}.
	 	 
	 	To complete the proof, let us verify that $P$ is an FF-monoid when the set of Archimedean classes $\{\beta(a_n) \mid n \in \nn\}$ contains infinitely many elements. Let $\{s_n\}$ be a strictly increasing sequence of natural numbers with $s_1 = 1$ so that $\beta(a_i) = \beta(a_j)$ if and only if $s_n \le i,j \le s_{n+1}-1$ for some $n \in \nn$. Set
	 	\[
		 	F_n = \langle a_1, \dots, a_{s_{n+1}-1} \rangle
	 	\]
	 	for every $n \in \nn$. By Proposition~\ref{prop:f.g. atomic monoids are FF-monoids}, each $F_n$ is an FF-monoid. Now we verify that each $F_n$ is a divisor-closed submonoid of $P$. If $y \in F_n$, then there are nonnegative integer coefficients $n_1, \dots, n_{s_{n+1}-1}$ such that
	 	\[
		 	y = \sum_{i=1}^{s_{n+1}-1} n_i a_i \le (s_{n+1} - 1)N a_{s_{n+1} - 1} < a_j	 
	 	\]
	 	for every $j \ge s_{n+1}$, where $N = \max\{n_1, \dots, n_{s_{n+1} - 1}\}$; this is because $\beta(a_j) \prec \beta(a_{s_{n+1} - 1})$ when $j \ge s_{n+1}$. Therefore no atoms of $P$ contained in the complement of $F_n$ divides $y$ in $P$. As a result, $F_n$ is a divisor-closed submonoid of $P$. Since $P$ is the union of the $F_n$'s, it follows by Lemma~\ref{lem:union of nested divisor-closed FF is FF} that $P$ is an FF-monoid. \hfill \qed

	We say that a monoid $M$ is \emph{hereditarily atomic} if each submonoid of $M$ is atomic. As the next proposition indicates, in the family of positive monoids, being hereditarily atomic is a consequence of being a BF-monoid.
	
	\begin{proposition} \label{prop:BF implies hereditarily atomicity}
		Every positive \emph{BF}-monoid of an ordered field is hereditarily atomic.
	\end{proposition}
	
	\begin{proof}
		Let $K$ be an ordered field, and let $P$ be a positive BF-monoid of $K$. In particular,~$P$ is atomic. Let $P'$ be a submonoid of $P$. Observe that every element of $P'$ that cannot be written as a sum of two elements in $P'^\bullet$ belongs to $\mathcal{A}(P')$. Take $x \in P'^\bullet$. Since $P$ is a BF-monoid, $\mathsf{L}_{P}(x)$ is finite, and so there exists $N \in \nn$ such that $|z| < N$ for all $z \in \mathsf{Z}_P(x)$. Now let us write
		\begin{equation} \label{eq:BF implies hereditarily atomicity}
			x = x'_1 + \dots + x'_n
		\end{equation}
		for some $n \in \nn$ and $x'_1, \dots, x'_n \in P'^\bullet$. Since each $x'_i$ belongs to $P^\bullet$, it follows that $x$ can be written as the sum of at least $n$ atoms of $P$. This implies that $n \le N$, and so~$x$ can be expressed as the sum of at most $N$ elements of $P'^\bullet$. Then we can choose~$n$ in \eqref{eq:BF implies hereditarily atomicity} to be maximal. In this case, each $x'_i$ must be an atom of $P'$. This implies that~$P'$ is atomic. Because the submonoid $P'$ of $P$ was arbitrarily taken, $P$ happens to be hereditarily atomic.
	\end{proof}
	
	The converse of Proposition~\ref{prop:BF implies hereditarily atomicity} does not hold. For example, according to \cite[Theorem~5.5]{GG16}, if $\{p_n\}$ is an enumeration of the prime numbers, then the Puiseux monoid $P = \langle 1/p_n \mid n \in \nn \rangle$ is hereditarily atomic. However, for every $n \in \nn$, the element $1$ is the sum of $p_n$ copies of the atom $1/p_n$ and, therefore, $P$ is not a BF-monoid. On the other hand, the positive condition on Proposition~\ref{prop:BF implies hereditarily atomicity} is not superfluous. For instance, the multiplicative monoid of the polynomial ring $\qq[X]$ is factorial and, in particular, a BF-monoid; however, its submonoid $\zz^\bullet + \qq[X]$ is not atomic (see \cite[Example~3]{sC14}) Finally, an atomic positive monoid that fails to be a BF-monoid might not be hereditarily atomic. The next example sheds light upon this.
	
	\begin{example}
		Let $\{p_n\}$ be a strictly increasing sequence comprising the odd prime numbers. Consider the Puiseux monoid of $\qq$
		\[
			P = \langle A \rangle, \text{ where } A = \bigg\{ \frac 1{2^np_n} \ \bigg{|} \ n \in \nn \bigg\}.
		\]
		Since each odd prime divides exactly one element of the set $\mathsf{d}(A)$, it follows that $\mathcal{A}(P) = A$. Thus, $P$ is atomic. Moreover, the fact that $1$ is the sum of $2^np_n$ copies of the atom $1/(2^np_n)$ for every $n \in \nn$ implies that $P$ is not a BF-monoid. On the other hand, the element $1/2^n$ is the sum of $p_n$ copies of the atom $1/(2^np_n)$ for every $n \in \nn$ and, therefore, the antimatter monoid $\langle 1/2^n \mid n \in \nn \rangle$ is a submonoid of $P$. Hence $P$ fails to be hereditarily atomic.
	\end{example}
	
	Combining Theorem~\ref{thm:increasing PM of Archimedean fields are FF} and Proposition~\ref{prop:BF implies hereditarily atomicity} we immediately obtain the next result.
	
	\begin{cor} \label{cor:strongly increasing PM are hereditarily atomic}
		Every increasing positive monoid of an ordered field is hereditarily atomic.
	\end{cor}
	
	\section*{Acknowledgments}
	
	While working on this paper, the author was supported by the UC Berkeley Chancellor Fellowship. The author is happy to thank the anonymous referee for the careful reading and valuable suggestions.

\end{document}